\documentclass[11pt]{article}

\usepackage[margin=1in]{geometry}
\usepackage[T1]{fontenc}
\DeclareUnicodeCharacter{2115}{\ensuremath{\mathbb{N}}}   
\DeclareUnicodeCharacter{2264}{\ensuremath{\le}}          
\DeclareUnicodeCharacter{2227}{\ensuremath{\wedge}}       
\DeclareUnicodeCharacter{2200}{\ensuremath{\forall}}      
\DeclareUnicodeCharacter{2192}{\ensuremath{\to}}          
\usepackage{amsmath,amssymb,amsthm,mathtools}
\usepackage{enumitem}
\usepackage{xcolor}
\definecolor{midnightblue}{RGB}{25,25,112}
\usepackage[colorlinks=true,allcolors=midnightblue]{hyperref}
\usepackage[backend=biber,style=numeric,giveninits=true,url=false,
            maxbibnames=99]{biblatex}
\DeclareSourcemap{%
  \maps[datatype=bibtex]{%
    \map[overwrite]{%
      \step[fieldsource=author,
            match=\regexp{\AThe\smathlib\sCommunity\z},
            replace=\regexp{\{The\x20mathlib\x20Community\}}]%
    }%
    \map[overwrite]{%
      \step[fieldsource=doi,
            match=\regexp{\Ahttps?://(dx\.)?doi\.org/},
            replace={}]%
    }%
    \map[overwrite]{%
      \step[fieldsource=entrykey,
            match=\regexp{\Akolokolnikov2014maximizingalgebraicconnectivitycertain\z},
            final]%
      \step[fieldset=addendum,
            fieldvalue={Published in \emph{Linear Algebra and its
              Applications} \textbf{471} (2015), 122--140,
              \url{https://doi.org/10.1016/j.laa.2014.12.023}}]%
    }%
    \map[overwrite]{%
      \step[fieldsource=entrykey,
            match=\regexp{\Agao2025semanticsearchenginemathlib4\z},
            final]%
      \step[fieldset=addendum,
            fieldvalue={Published in \emph{Findings of the Association for
              Computational Linguistics: EMNLP 2024}, Association for
              Computational Linguistics, pp.~8001--8013,
              \url{https://doi.org/10.18653/v1/2024.findings-emnlp.470}}]%
    }%
  }%
}
\numberwithin{equation}{section}

\newtheorem{theorem}{Theorem}[section]
\newtheorem{lemma}[theorem]{Lemma}
\newtheorem{proposition}[theorem]{Proposition}
\newtheorem{corollary}[theorem]{Corollary}
\newtheorem{conjecture}[theorem]{Conjecture}
\theoremstyle{definition}
\newtheorem{definition}[theorem]{Definition}
\theoremstyle{remark}

\newcommand{\ac}{\lambda_2}

\title{Maximizing Algebraic Connectivity with \(2(n-2)\) Edges}

\author{
Zeru Zhu \and Jinzheng Li \and Yuanjie Ren \and Ji Liu
}

\date{}

\begin{document}

\maketitle
\begingroup
\renewcommand{\thefootnote}{}
\footnotemark
\footnotetext{Zeru Zhu: Department of Applied Mathematics \& Statistics,
Stony Brook University, \texttt{zeru.zhu@stonybrook.edu}; Jinzheng Li:
Department of Physics, Northeastern University,
\texttt{li.jinzh@northeastern.edu}; Yuanjie Ren: Department of Physics,
Massachusetts Institute of Technology, \texttt{yuanjie@mit.edu}; Ji Liu:
Department of Electrical and Computer Engineering, Stony Brook University,
\texttt{ji.liu@stonybrook.edu}. The first three authors contributed equally.}
\endgroup
\setcounter{footnote}{0}

\begin{abstract}
Kolokolnikov conjectured that, among all simple graphs on \(n\) vertices
with exactly \(2(n-2)\) edges, the complete bipartite graph 
maximizes algebraic connectivity.  This paper proves the conjecture. 
The underlying Lean~4 formalization was
generated with MerLean and checked by the Lean kernel.
\end{abstract}



\section{Introduction}\label{sec:introduction}

For a finite simple graph \(G\) on \(n\) vertices, let \(A(G)\) and
\(D(G)\) be its adjacency and diagonal degree matrices, respectively, and
let \(L(G)=D(G)-A(G)\) be its Laplacian matrix.  Order the Laplacian
eigenvalues as
\(0=\lambda_1(G)\le\lambda_2(G)\le\cdots\le\lambda_n(G)\).  The
second-smallest Laplacian eigenvalue \(\lambda_2(G)\) is called the
\emph{algebraic connectivity} of \(G\).  Introduced by
Fiedler~\cite{fiedler1973}, this parameter relates spectral information to
the connectivity and expansion of a graph.  A basic extremal problem asks
how large \(\ac(G)\) can be when the order and size of \(G\) are prescribed.

For positive integers \(p\) and \(q\), let \(K_{p,q}\) denote the complete
bipartite graph with vertex classes of sizes \(p\) and \(q\).
The graph \(K_{2,n-2}\) has \(n\) vertices, \(2(n-2)\) edges, and algebraic
connectivity \(2\), as verified in Lemma~\ref{lem:bipartite-spectrum} below.
Kolokolnikov conjectured that no graph with the same order and size has
larger algebraic connectivity~\cite[Conjecture~1.5]{kolokolnikov2014maximizingalgebraicconnectivitycertain}.
Computational searches for small orders reported in the same paper support
this conjecture.

\begin{conjecture}[Kolokolnikov]
\label{conj:kolokolnikov}
Let \(n\ge4\).  Every finite simple graph \(G\) on \(n\) vertices
with \(2(n-2)\) edges has \(\ac(G)\le2\), with equality attained by
\(K_{2,n-2}\).
\end{conjecture}

Our main result proves the conjecture in full.

\begin{theorem}\label{thm:main}
Let \(n\ge4\).  Every finite simple graph \(G\) on \(n\) vertices
with \(2(n-2)\) edges has \(\ac(G)\le2\).
\end{theorem}

\begin{samepage}
Combining Theorem~\ref{thm:main} with
Lemma~\ref{lem:bipartite-spectrum} gives the extremal conclusion.

\begin{corollary}\label{cor:extremal}
For every \(n\ge4\), the maximum algebraic connectivity among finite
simple graphs on \(n\) vertices with \(2(n-2)\) edges is \(2\), and
\(K_{2,n-2}\) is a maximizer.
\end{corollary}
\end{samepage}

Corollary~\ref{cor:extremal} proves maximality, not uniqueness; no
classification of the equality cases is asserted here.

Previous work on algebraic-connectivity maximization with fixed order and size
gives sufficient conditions for complete bipartite and, more generally,
complete multipartite graphs to be extremal
\cite{ogiwara2017maximizing,shahbaz2023algebraic}.  These criteria include
verifications of the present conjecture through \(n=8\), but do not cover all
orders.

The first version of this paper \cite{zhu2026largeorder} gave a self-contained proof for \(n\ge123\) and reported a Lean formalization for \(n\ge4\), while deferring a natural language treatment of \(4\le n\le122\). The current version gives a complete self-contained mathematical proof for \(n\ge4\). Three days after our first version appeared, two groups of researchers, Chi, Wang, and Zheng \cite{chi2026kolokolnikov} and Cioabă et al. \cite{cioaba2026kolokolnikov}, also posted preprints claiming complete proofs of the Kolokolnikov conjecture. The three papers follow different overall proof strategies.  

A Lean formalization of Conjecture~\ref{conj:kolokolnikov} for every
\(n\ge4\) has been produced with MerLean \cite{li2026merleanproverrecursiveloopingharness,ren2026merleanagenticframeworkautoformalization}%
\footnote{\url{https://github.com/MerLeanProver/MerLean}} and checked by the
Lean kernel\footnote{\url{https://github.com/MerLeanProver/ACMaxConjecture}}.
The present paper reorganizes that development as a self-contained
mathematical proof.  Appendix~\ref{sec:formalization} records the formal
verification and summarizes the MerLean system; a full account of its
architecture and evaluation will be given separately.

The proof begins in Section~\ref{sec:certificates} with explicit
Rayleigh-quotient test vectors that impose strong local restrictions on a
hypothetical counterexample.  Section~\ref{sec:census} translates those
restrictions into global degree and incidence inequalities.  For \(n\ge48\),
Sections~\ref{sec:cycles} and~\ref{sec:obstruction-closure} combine these counts
with an exact non-backtracking Moore bound to force a short cycle that the
spectral certificates forbid.  Section~\ref{sec:obstruction-closure} also uses
a sharper incidence count for \(32\le n\le49\), while
Section~\ref{sec:orders-below-thirty-two} completes \(4\le n\le31\) with local
sparse-set and cut arguments.

\section{Spectral certificates and the obstruction class}
\label{sec:certificates}

The proof proceeds by contradiction.  This section
develops the Rayleigh-quotient certificates used to exclude vertices of
degree \(\le2\), adjacent vertices of degree \(3\), and several other
local configurations.  Together, these restrictions reduce a hypothetical
counterexample to the degree-\(3\)-separated obstruction defined at the end
of the section.

Throughout the proof, \(G=(V,E)\) is a finite simple graph and \(n=|V|\).
For \(v\in V\), write \(N(v)\) for its neighborhood,
\(d(v)=|N(v)|\) for its degree, and
\(\delta(G)=\min_{v\in V}d(v)\).  For \(U\subseteq V\), let
\(N(U)=\bigcup_{v\in U}N(v)\), let \(G[U]\) be the subgraph induced by
\(U\), and let \(e(U)\) be the number of edges with both ends in \(U\).
If \(A,B\subseteq V\) are disjoint, then \(e(A,B)\) denotes the number of
edges with one end in \(A\) and the other in \(B\).

\subsection{Variational and separation certificates}

\begin{lemma}[The extremal example]\label{lem:bipartite-spectrum}
For every \(n\ge4\), the graph \(K_{2,n-2}\) has algebraic connectivity~\(2\).
\end{lemma}

\begin{proof}
Let \(U,W\) be the two vertex classes, with \(|U|=2\) and \(|W|=n-2\).
The vectors supported on \(W\) and summing to zero form an
\((n-3)\)-dimensional subspace of eigenvectors with eigenvalue \(2\).  The
constant vector
has eigenvalue \(0\); the zero-sum vectors supported on \(U\) give the
eigenvalue \(n-2\); and the remaining one-dimensional space, consisting of
vectors constant on each class and orthogonal to the constant vector, has
eigenvalue \(n\).  Thus the second Laplacian eigenvalue is \(2\).
\end{proof}

By the Courant--Fischer theorem (see, for example,
\cite{brouwerhaemers2012}), the algebraic connectivity satisfies
\begin{equation}\label{eq:variational}
  \ac(G)
  =
  \min_{\substack{x\in\mathbb{R}^{V}\setminus\{0\}\\
                  \sum_{v\in V}x_v=0}}
  \frac{\displaystyle\sum_{uv\in E}(x_u-x_v)^2}
       {\displaystyle\sum_{v\in V}x_v^2}.
\end{equation}
Here \(x\) is one real vector indexed by the vertices of \(G\), and \(x_v\)
is its coordinate at \(v\).  The numerator is the Dirichlet energy of
\(x\): only edges whose endpoints receive different values contribute.
Our certificates therefore use vectors that are constant on large vertex
sets and change value only across controlled boundaries.

In particular, any nonzero vector \(x\) with coordinate sum zero and
\(\sum_{uv\in E}(x_u-x_v)^2\le2\sum_{v\in V}x_v^2\) certifies that
\(\ac(G)\le2\).  We use this observation repeatedly.

The following proposition isolates the Rayleigh-quotient argument used in all
subsequent local reductions.  The set \(F\) separates two nonempty sets
\(A\) and \(B\).  For connected \(G\), the proposition is the special case of
a more general inequality of Liu et al.\ in which no edge joins \(A\) and
\(B\) directly~\cite[Lemma~3.2]{liu2014}; we include the short direct proof
in the notation used throughout this paper.

\begin{proposition}[Two-cluster certificate]\label{lem:two-cluster}
Let \(A,B\subseteq V\) be disjoint nonempty sets with no edge between
them, and put \(F=V\setminus(A\cup B)\).  If \(e(A,F)\le2|A|\) and
\(e(B,F)\le2|B|\), then \(\ac(G)\le2\).
\end{proposition}

\begin{proof}
Write \(a=|A|\) and \(b=|B|\), and define
\[
  x_v=
  \begin{cases}
    b,  & v\in A,\\
    -a, & v\in B,\\
    0,  & v\in F.
  \end{cases}
\]
Then \(\sum_v x_v=0\) and \(\sum_v x_v^2=ab(a+b)\).  There are no edges
between \(A\) and \(B\), so the numerator is
\(b^2e(A,F)+a^2e(B,F)\).  By the two boundary assumptions, this is
\(\le2ab^2+2a^2b=2ab(a+b)\), twice the squared norm of \(x\), and
\eqref{eq:variational} gives \(\ac(G)\le2\).
\end{proof}

We next record the sparse-set argument used in the first local reduction.
For \(S\subseteq V\), put
\[
 I(S)=\sum_{v\in S}|N(v)\cap S|=2e(S),
 \qquad \varepsilon(S)=\sum_{v\in S}(d(v)-3).
\]

\begin{definition}[Two-sparse set]\label{def:two-sparse}
A nonempty set \(S\subseteq V\) is \emph{two-sparse} if
\(|N(v)\setminus S|\le2\) for every \(v\in S\).
\end{definition}

\begin{proposition}[Sparse-set principle]\label{prop:sparse-set-principle}
The following statements hold.
\begin{enumerate}[label=\textup{(\roman*)}]
  \item Two disjoint two-sparse sets with no edge between them certify
  \(\ac(G)\le2\).
  \item If \(\delta(G)\ge3\) and a vertex set \(S\) satisfies
  \(I(S)>2\varepsilon(S)\), then \(S\) contains a two-sparse set.
\end{enumerate}
\end{proposition}

\begin{proof}
For (i), let the two sets be \(S,T\), and put
\(F=V\setminus(S\cup T)\).  Two-sparsity gives
\(e(S,F)\le2|S|\) and \(e(T,F)\le2|T|\), so
Proposition~\ref{lem:two-cluster} applies.

For (ii), suppose otherwise and start with \(T_0=S\).  As long as \(T_i\) is
nonempty, choose \(v_i\in T_i\) with at least three neighbors outside
\(T_i\), and put \(T_{i+1}=T_i\setminus\{v_i\}\).  The process eventually
deletes every vertex of \(S\).  Removing \(v_i\) decreases the ordered
internal-edge count by \(2|N(v_i)\cap T_i|\le2(d(v_i)-3)\).  Telescoping over
the deletion sequence gives \(I(S)\le2\varepsilon(S)\), a contradiction.
\end{proof}

\subsection{Initial reductions and the degree budget}

We now begin the structural reduction.  The first certificate rules out low
degree, after which the fixed edge count supplies the degree budget used in
the remaining local arguments.

\begin{lemma}[Low-degree vertex]\label{lem:low-degree}
Let \(u\) be a vertex of degree \(\le2\).  If
\(T=V\setminus(\{u\}\cup N(u))\) is nonempty, then \(\ac(G)\le2\).
\end{lemma}

\begin{proof}
Put \(t=|T|\), and set \(x_u=t\), \(x_v=-1\) for \(v\in T\), and
\(x_v=0\) for \(v\in N(u)\).  The vector has coordinate sum zero and
squared norm \(t^2+t\).  There are no edges from \(u\) to \(T\).  The
contribution from edges joining \(u\) to \(N(u)\) is \(\le2t^2\), and each
vertex of \(T\) has \(\le2\) neighbors in \(N(u)\), whose total contribution
is \(\le2t\).  Thus the Rayleigh numerator is \(\le2(t^2+t)\), and
\eqref{eq:variational} gives \(\ac(G)\le2\).
\end{proof}

The fixed edge count gives average degree \(4-8/n\), so degree \(4\) is
the natural reference level.  Once minimum degree \(3\) is established,
degree-\(3\) vertices are the only source of deficit below this level,
while vertices of degree at least \(5\) provide the compensating excess.
Indeed, the degree-sum identity gives
\(\sum_{v\in V}d(v)=2|E|=4n-8\), and hence
\begin{equation}\label{eq:total-excess}
  \sum_{v\in V}\bigl(d(v)-3\bigr)=n-8.
\end{equation}
When \(\delta(G)\ge3\), every summand is nonnegative.  Under this assumption,
we will repeatedly use the following boundary estimate.  Given a nonempty
set \(A\), put \(F=N(A)\setminus A\) and
\(B=V\setminus(A\cup F)\).  Every vertex of \(F\) has a neighbor in \(A\), so
\begin{equation}\label{eq:separator-ledger}
  e(B,F)
  \le \sum_{w\in F}(d(w)-1)
  =2|F|+\varepsilon(F)
  \le2|F|+n-8-\varepsilon(A).
\end{equation}

\begin{lemma}[Adjacent degree-\(3\) vertices]\label{lem:adjacent-three}
Assume that \(n\ge10\), that \(G\) has \(2(n-2)\) edges, and that
\(\delta(G)\ge3\).  If two degree-\(3\) vertices are adjacent, then
\(\ac(G)\le2\).
\end{lemma}

\begin{proof}
Let \(A=\{u,v\}\), let \(F=N(A)\setminus A\), and put
\(B=V\setminus(A\cup F)\).  Then \(|F|\le4\), the set \(A\) is
two-sparse, and there is no edge from \(A\) to \(B\).  Write \(f=|F|\),
\(b=|B|=n-2-f\), and \(q=e(B,F)\).  Since \(n\ge10\) and \(f\le4\), we
have \(b\ge4\).  The first inequality in \eqref{eq:separator-ledger} gives
\(q\le2f+\varepsilon(F)\), while counting degrees in \(B\) gives
\(q+I(B)=3b+\varepsilon(B)\).
Since \(u\) and \(v\) have zero excess,
\(\varepsilon(F)+\varepsilon(B)=n-8\).  The two preceding bounds therefore
imply
\[
 I(B)\ge3b+\varepsilon(B)-\bigl(\varepsilon(F)+2f\bigr)
       =2\varepsilon(B)+2n+2-5f>2\varepsilon(B),
\]
where the last inequality uses \(n\ge10\) and \(f\le4\).
Proposition~\ref{prop:sparse-set-principle}(ii) supplies a two-sparse set
\(T\subseteq B\).  The sets \(A\) and \(T\) are separated, so part~(i)
completes the proof.
\end{proof}

Suppose that \(n\ge10\) and that \(G\) is a counterexample to
Conjecture~\ref{conj:kolokolnikov}.  If a vertex had degree \(\le2\),
then its closed neighborhood would contain \(\le3\) vertices, so
Lemma~\ref{lem:low-degree} would contradict \(\ac(G)>2\).  Hence
\(\delta(G)\ge3\), and Lemma~\ref{lem:adjacent-three} then implies that
the degree-\(3\) vertices form an independent set.  Thus every
counterexample in this range belongs to the following obstruction class.

\begin{definition}[Degree-\(3\)-separated obstruction]
\label{def:separated-obstruction}
A finite simple graph \(G\) on \(n\) vertices is called a
\emph{degree-\(3\)-separated obstruction} if it has \(2(n-2)\) edges,
\(\ac(G)>2\), \(\delta(G)\ge3\), and the set
\(D_3(G)=\{v\in V(G):d(v)=3\}\) is independent.
\end{definition}

\subsection{Further local certificates}

Having identified the obstruction class containing every possible
counterexample, we record three further configurations excluded from it.
The lemmas are stated under only the hypotheses they need.  Their proofs all
follow the same pattern: the local configuration forms \(A\), its external
neighborhood forms the separator \(F\), and the remaining vertices form
\(B\).  Local degree information controls \(e(A,F)\) and \(|F|\), while
\eqref{eq:separator-ledger} controls \(e(B,F)\); the two-cluster certificate
applies once \(e(A,F)\le2|A|\) and \(e(B,F)\le2|B|\).

\begin{lemma}[Star certificate]\label{lem:star}
Assume that \(G\) has \(2(n-2)\) edges and \(\delta(G)\ge3\).  Let \(z\)
have degree \(d\ge4\) and at least \(d-2\) distinct degree-\(3\)
  neighbors.  If
\(9d\le n+15\),
then \(\ac(G)\le2\).
\end{lemma}

\begin{proof}
Choose \(d-2\) degree-\(3\) neighbors of \(z\), call their set \(K\), and
put \(A=\{z\}\cup K\).  Let \(F\) be the external neighborhood of \(A\)
and \(B=V\setminus(A\cup F)\).  Since \(|A|=d-1\), and both \(z\) and each
vertex of \(K\) have \(\le2\) edges leaving \(A\), we have
\(e(A,F)\le2|A|\) and \(|F|\le2d-2\).
The excess \(\varepsilon(A)\) is at least \(d-3\), so
\eqref{eq:separator-ledger} gives
\(e(B,F)\le2|F|+n-8-(d-3)\le n+3d-9\).  Moreover,
\(|B|\ge n-(d-1)-(2d-2)=n-3d+3\).  The hypothesis \(9d\le n+15\) is
equivalent to
\(n+3d-9\le2(n-3d+3)\), and it also ensures that \(B\) is nonempty.
Thus Proposition~\ref{lem:two-cluster} applies.
\end{proof}

\begin{lemma}[Adjacent high-degree vertices]\label{lem:decorated-edge}
Assume that \(G\) has \(2(n-2)\) edges and \(\delta(G)\ge3\).  Let \(u,v\)
be adjacent vertices of degrees \(d_u,d_v\ge4\), and put
\(D=d_u+d_v\).  Suppose that
\(K_u\subseteq N(u)\) and \(K_v\subseteq N(v)\) consist of degree-\(3\)
vertices, with \(|K_u|=d_u-3\) and \(|K_v|=d_v-3\).  Each of the
following conditions implies \(\ac(G)\le2\):
\begin{enumerate}[label=\textup{(\roman*)}]
  \item \(K_u\cap K_v=\varnothing\) and \(9D\le n+42\);
  \item \(|K_u\cap K_v|=1\) and \(9D\le n+56\).
\end{enumerate}
\end{lemma}

\begin{proof}
Let \(A=\{u,v\}\cup K_u\cup K_v\), let \(F\) be the external neighborhood
of \(A\), and put \(B=V\setminus(A\cup F)\).

In case (i), \(|A|=D-4\).  Each of \(u,v\), and each degree-\(3\) vertex
in \(K_u\cup K_v\), has \(\le2\) edges leaving \(A\).  Consequently,
\(e(A,F)\le2|A|\) and \(|F|\le2D-8\).  The excess of \(A\) is at least
\(D-6\), so \eqref{eq:separator-ledger} gives
\[
 e(B,F)\le n+3D-18,
 \qquad |B|\ge n-3D+12.
\]
The hypothesis \(9D\le n+42\) is exactly the comparison
\(n+3D-18\le2(n-3D+12)\).

In case (ii), \(|A|=D-5\).  The common degree-\(3\) vertex has two
neighbors in \(A\), and therefore contributes \(\le1\), rather than
two, to the external-edge budget.  Thus \(e(A,F)\le2|A|-1\) and
\(|F|\le2D-11\).  Since \(\varepsilon(A)\ge D-6\),
\eqref{eq:separator-ledger} gives
\[
 e(B,F)\le n+3D-24,
 \qquad |B|\ge n-3D+16.
\]
Here the hypothesis \(9D\le n+56\) is exactly the comparison
\(n+3D-24\le2(n-3D+16)\).  In both cases the numerical hypothesis also
makes the displayed lower bound for \(|B|\) positive.  Thus
\(e(B,F)\le2|B|\), and Proposition~\ref{lem:two-cluster} applies.
\end{proof}

\begin{lemma}[Cycle certificate]\label{lem:cycle-certificate}
Assume that \(G\) has \(2(n-2)\) edges and \(\delta(G)\ge3\).  Let \(C\)
be a cycle of length \(k\ge3\).  If
\(\sum_{v\in V(C)}d(v)\le4k\) and
\(3\sum_{v\in V(C)}(d(v)-1)\le n+8\),
then \(\ac(G)\le2\).
\end{lemma}

\begin{proof}
Put \(A=V(C)\), let \(F\) be the external neighborhood of \(A\), set
\(B=V\setminus(A\cup F)\), and write
\(D_C=\sum_{v\in A}d(v)\).  Every cycle vertex already has two
neighbors in \(A\), and therefore \(e(A,F)\le D_C-2k\le2k=2|A|\) and
\(|F|\le D_C-2k\).
The excess of \(A\) is \(D_C-3k\), so \eqref{eq:separator-ledger} gives
\(e(B,F)\le n-8+D_C-k\).  Also \(|B|\ge n+k-D_C\).  The second
hypothesis says \(3(D_C-k)\le n+8\), which is exactly
\(n-8+D_C-k\le2(n+k-D_C)\).
Finally, \(D_C-k\ge2k\), so the same hypothesis implies
\(n\ge6k-8\).  Together with \(D_C\le4k\), this gives
\(|B|\ge n-3k\ge3k-8>0\).  The two-cluster certificate completes the
proof.
\end{proof}

In a degree-\(3\)-separated obstruction, the \(3|D_3(G)|\) edges leaving
\(D_3(G)\) all end at vertices of degree \(\ge4\).  The next section compares
their number with the total degree available at those vertices, producing a
global counting constraint and an edge-excess bound for the low-degree
induced subgraph.

\section{Census of a separated obstruction}\label{sec:census}

Let \(G\) be a degree-\(3\)-separated obstruction.  Write
\(D_3=D_3(G)\), set \(S=\{v\in V:d(v)\le4\}\), and
\(R=V\setminus S=\{v\in V:d(v)\ge5\}\).  We call the vertices in
\(R\) \emph{heavy}.  Write \(h=|R|\),
\(X=\sum_{v\in R}(d(v)-4)\), and
\(\tau(v)=|N(v)\cap D_3|\).  Thus \(X\) is the total degree excess above
\(4\) among the vertices of degree at least \(5\), while \(\tau(v)\)
counts the degree-\(3\) neighbors of \(v\).

It is useful to count the edges leaving \(D_3\) in two ways.  From the
\(D_3\) side, their number is fixed by \(|D_3|\); at their other endpoints,
they must be accommodated by vertices of degree at least \(4\).  The quantity
\(X\) plays two competing roles: it increases the number of degree-\(3\)
vertices, but it also allows high-degree vertices to have more neighbors in
\(D_3\).  The count compares the number of these edges with the available
degree capacity.

\subsection{Degree identities and local capacity}

Splitting the degree sum into degrees \(3\), \(4\), and at least \(5\) gives
\(4n-8=4n-|D_3|+X\).  Together with the independence of \(D_3\), this yields
\begin{align}
  |D_3|&=8+X,                                      \label{eq:d3-count}\\
  \sum_{v:\,d(v)\ge4}\tau(v)&=24+3X,             \label{eq:incidence-total}\\
  h&\le X,                                         \label{eq:h-le-X}\\
  |S|&=n-h,                                        \notag\\
  \sum_{v:\,d(v)\ge4}d(v)&=4n-32-3X.            \label{eq:high-degree-sum}
\end{align}
Indeed, the second identity counts the edges leaving \(D_3\) in two ways, the
third and fourth follow from the definitions of \(X,h,S\), and the last
subtracts the degree contribution \(3|D_3|=24+3X\) from \(4n-8\).

Equation~\eqref{eq:incidence-total} counts the edges leaving \(D_3\).  We
next bound the contribution of each vertex to this total.  For a vertex
of degree at least \(4\), the star certificate controls \(\tau(v)\)
whenever \(9d(v)\le n+15\).  We therefore call \(v\)
\emph{exceptional} if \(9d(v)>n+15\).

\begin{lemma}[Degree-\(3\)-neighbor cap and independence]\label{lem:incidence-cap}
Let \(G\) be a degree-\(3\)-separated obstruction of order \(n\ge48\).
\begin{enumerate}[label=\textup{(\roman*)}]
  \item For every \(v\in V\), if \(9d(v)\le n+15\), then
  \(\tau(v)\le d(v)-3\).
  \item A degree-\(4\) vertex has \(\le1\) neighbor in \(D_3\).
  \item The sets \(O=\{v:d(v)=4,\ \tau(v)=1\}\) and
  \(P=\{v:d(v)=5,\ \tau(v)=2\}\) have the property that \(O\cup P\) is
  independent.
\end{enumerate}
\end{lemma}

\begin{proof}
If \(d(v)=3\), then \(\tau(v)=0\) because \(D_3\) is independent.  If a
vertex \(v\) of degree at least \(4\) satisfying \(9d(v)\le n+15\) had
at least \(d(v)-2\) degree-\(3\) neighbors, the star certificate would
contradict \(\ac(G)>2\).  This proves (i).  A degree-\(4\) vertex is
nonexceptional for \(n\ge48\), so (ii) follows.

It remains to prove (iii).  Suppose first that two vertices of \(O\) are
adjacent.  If they share their unique degree-\(3\) neighbor, they lie in a
triangle with degree pattern \((4,4,3)\), and
Lemma~\ref{lem:cycle-certificate} applies for \(n\ge16\).  If their degree-\(3\)
neighbors are distinct, Lemma~\ref{lem:decorated-edge}(i) applies with
degree sum \(8\), requiring \(n\ge30\).

Next suppose that a vertex of \(O\) is adjacent to a vertex of \(P\).
A shared degree-\(3\) neighbor produces a triangle with degree pattern
\((4,5,3)\), covered by Lemma~\ref{lem:cycle-certificate} for \(n\ge19\).
Otherwise Lemma~\ref{lem:decorated-edge}(i) applies with degree sum \(9\)
for \(n\ge39\).

Finally, suppose that two vertices of \(P\) are adjacent.  Their sets of
degree-\(3\) neighbors intersect in zero, one, or two vertices.  The first
two cases are covered by parts (i) and (ii), respectively, of
Lemma~\ref{lem:decorated-edge}, with degree sum \(10\); the corresponding
requirements are \(n\ge48\) and \(n\ge34\).  In the last case,
the two common degree-\(3\) neighbors and the two degree-\(5\) vertices
form a \(4\)-cycle of degree pattern \((5,3,5,3)\), and
Lemma~\ref{lem:cycle-certificate} applies for \(n\ge28\).  All these
requirements hold when \(n\ge48\), so every possible adjacency in
\(O\cup P\) contradicts \(\ac(G)>2\).
\end{proof}

For any degree-\(3\)-separated obstruction, define
\begin{align*}
  \tau_4&=\sum_{v:\,d(v)=4}\tau(v),&
  p&=\bigl|\{v:d(v)=5,\ \tau(v)=2\}\bigr|,\\
  \rho&=\bigl|\{v:d(v)\ge6,\ 9d(v)\le n+15\}\bigr|,&
  g&=\bigl|\{v:9d(v)>n+15\}\bigr|.
\end{align*}
Here \(\tau_4\) records all incidences at degree \(4\).  Relative to the
baseline \(d(v)-4\), a vertex counted by \(p\) or \(\rho\) may contribute
one additional incidence, while the trivial bound for an exceptional
vertex may contribute four.

For an obstruction of order \(n\ge32\), two elementary estimates will be
used in both range arguments of Section~\ref{sec:obstruction-closure}.  A
nonexceptional vertex contributes
\(\le d(v)-3\) to the incidence total, by the star certificate, whereas the
trivial bound for an exceptional vertex costs three additional units.  Thus
\begin{align}
  24+3X
  &\le\sum_{v:\,d(v)\ge4}(d(v)-3)+3g=n-8+3g,\notag\\
  3X+32&\le n+3g.                    \label{eq:exceptional-hoarding}
\end{align}
Every exceptional vertex satisfies \(n-20\le9(d(v)-4)\) and, since
\(n\ge32\), has degree at least \(6\).  Therefore
\begin{equation}\label{eq:exceptional-giants}
  g(n-20)
  \le9\sum_{v:\,9d(v)>n+15}(d(v)-4)
  \le9X.
\end{equation}

\subsection{Global inequalities and low-degree density}

We first derive the global census inequality needed by the exact Moore
argument.  Let the obstruction have order \(n\ge48\).  Regard \(d(v)-4\) as
the baseline incidence contribution of each \(v\in R\).  The contribution
beyond baseline is \(\le1\) for each vertex counted by \(p\) or \(\rho\),
and \(\le4\) for each vertex counted by \(g\).  Thus
\(24+3X\le X+\tau_4+p+\rho+4g\), and hence
\begin{equation}\label{eq:slots}
  24+2X\le\tau_4+p+\rho+4g.
\end{equation}
On the other hand, Lemma~\ref{lem:incidence-cap}(ii) gives
\(|O|=\tau_4\), while \(|P|=p\) by definition.  The independence of
\(O\cup P\) forces its members to send \(3\tau_4+3p\) edges to vertices of
degree at least \(4\) outside the set.  Consequently,
\(\sum_{v:\,d(v)\ge4}d(v)
\ge(4\tau_4+5p)+(3\tau_4+3p)=7\tau_4+8p\).
Together with \eqref{eq:high-degree-sum}, this gives
\begin{equation}\label{eq:capacity}
  7\tau_4+8p+3X+32\le4n.
\end{equation}

The auxiliary counts can now be eliminated to obtain the numerical
constraint used by the Moore argument.

\begin{lemma}[Moore-range census]\label{lem:moore-census}
Let \(G\) be a degree-\(3\)-separated obstruction of order \(n\ge48\).  Then
\begin{equation}\label{eq:moore-census}
  10X+7h+186\le4n.
\end{equation}
\end{lemma}

\begin{proof}
If \(g\ge3\), inequalities \eqref{eq:exceptional-hoarding} and
\eqref{eq:exceptional-giants} imply
\((g-3)n\le29g-96\).  But for \(g\ge3\),
\(48(g-3)-(29g-96)=19g-48>0\), contradicting \(n\ge48\).  Thus \(g\le2\).

Every heavy vertex contributes at least one unit to \(X\), every vertex
counted by \(\rho\) contributes one further unit, and every exceptional
vertex contributes three further units because it has degree at least
\(8\).  Since the sets counted by \(\rho\) and \(g\) are disjoint,
\begin{equation}\label{eq:heavy-ledger}
  h+\rho+3g\le X.
\end{equation}
Equations \eqref{eq:slots} and \eqref{eq:capacity} give, after multiplying the
former by \(7\) and comparing the resulting \(7\tau_4+7p\) term with the latter,
\(17X+p\le4n-200+7\rho+28g\).  Together with
\eqref{eq:heavy-ledger}, this yields
\[
  10X+7h\le17X-7\rho-21g
  \le4n-200+7g-p\le4n-186,
\]
where the three comparisons use, respectively, \eqref{eq:heavy-ledger}, the
preceding bound on \(17X+p\), and \(g\le2\) with \(p\ge0\).
\end{proof}

The second output of the global count concerns the subgraph induced by
\(S\).  The next calculation expresses its edge excess in terms of the
degree variables \(X\) and \(h\).

\begin{lemma}[Low-degree edge excess]\label{lem:S-density}
For a degree-\(3\)-separated obstruction, define \(t=n-4-X-3h\).
Then
\begin{equation}\label{eq:S-density}
  e(S)\ge|S|+t.
\end{equation}
\end{lemma}

\begin{proof}
The total degree of \(R=V\setminus S\) is \(4h+X\).  Hence
\(e(S,R)\le4h+X\).  Since the total degree
sum is \(4n-8\),
\[
\begin{aligned}
  2e(S)
  &=
  \sum_{v\in S}d(v)-e(S,R)\\
  &\ge
  4n-8-2(4h+X)\\
  &=
  2(n-h)+2(n-4-X-3h)\\
  &=
  2|S|+2t.
\end{aligned}
\]
Dividing by \(2\) proves the result.
\end{proof}

The census has therefore reduced the large-order argument to a tension between
density and girth: \eqref{eq:S-density} supplies an edge surplus in \(G[S]\),
while \eqref{eq:moore-census} and \(h\le X\) control its parameters.  The next
section turns this surplus into a short cycle by an exact Moore bound and then
uses the spectral certificates to exclude a cycle of the same length.

\section{Exact Moore bounds and cycle exclusion}\label{sec:cycles}

We now formulate the two sides of this density--girth tension for the subgraph
induced by the vertices of degree \(\le4\).  An exact Moore bound converts its
edge excess into a short cycle, whereas the spectral certificates of
Section~\ref{sec:certificates} exclude every cycle in the same length range.

\subsection{Exact non-backtracking growth}

\begin{proposition}[Non-backtracking Moore inequality]\label{lem:nb-moore}
Let \(H\) be a finite simple graph with \(N\) vertices and \(\delta(H)\ge2\).
Write \(d_H(x)\) for the degree of \(x\) in \(H\), let
\(D_H=\sum_{x\in V(H)}d_H(x)\) be its degree sum, and let \(\ell\ge1\) be an
integer.  If \(H\) has no cycle of length \(\le2\ell\), then
\[
 D_H\sum_{i=0}^{\ell-1}(D_H-N)^iN^{\ell-1-i}\le N^\ell(N-1).
\]
\end{proposition}

\begin{proof}
Put \(\bar d=D_H/N\).  Since the girth of \(H\) is at least
\(2\ell+1\), the irregular Moore bound of Alon, Hoory, and
Linial~\cite{ahl2002} gives
\[
 N\ge 1+\bar d\sum_{i=0}^{\ell-1}(\bar d-1)^i.
\]
Here the odd-girth bound applies directly when the girth is odd; when it is
even, the corresponding even-girth bound is stronger.  Substituting
\(\bar d=D_H/N\), subtracting \(1\), and multiplying by \(N^\ell\) gives the
stated inequality.
\end{proof}

To apply Proposition~\ref{lem:nb-moore} to an induced subgraph with \(v\)
vertices and edge surplus at least \(t\), we use the numerical condition
\begin{equation}\label{eq:ball-condition}
 t(v-1)v^\ell<(v+t)\bigl((v+2t)^\ell-v^\ell\bigr).
\end{equation}
The next proposition shows that this condition survives passage to the
possibly smaller \(2\)-core and forces a short cycle.

\begin{proposition}[Induced-set Moore criterion]\label{prop:nb-induced-cycle}
Let \(G\) be a finite simple graph and let \(U\subseteq V(G)\) be nonempty.
Put \(v=|U|\), and suppose \(e(U)\ge v+t\) for an integer \(t\ge1\).  Let \(L\ge6\) and
\(\ell=\lfloor L/2\rfloor\).  If \eqref{eq:ball-condition} holds, then
\(G[U]\) contains a cycle of length \(\le L\).
\end{proposition}

\begin{proof}
Repeatedly delete vertices of degree \(\le1\) from \(G[U]\), and let
\(H\) be the resulting \(2\)-core.  Each deletion removes one vertex and
\(\le1\) edge, so \(e-|V|\) does not decrease.  Since initially
\(e(U)-v\ge t>0\), the process cannot delete every vertex.  Thus \(H\) is
nonempty and, if \(v'=|V(H)|\), then \(e(H)\ge v'+t\).  Its total degree
\(D\) therefore satisfies \(D\ge2(v'+t)\) and \(D-v'\ge v'+2t\).

After division by \(tv^\ell\), condition \eqref{eq:ball-condition} is
equivalent to
\[
 v-1<2\left(1+\frac tv\right)
 \sum_{i=0}^{\ell-1}\left(1+\frac{2t}{v}\right)^i.
\]
The left-hand side decreases and the right-hand side increases as \(v\)
decreases, because both \(t/v\) and \(2t/v\) increase.  Since \(v'\le v\),
the same strict inequality holds with \(v'\) in place of \(v\).

Suppose that \(H\) has no cycle of length \(\le L\).  Since
\(2\ell\le L\), Proposition~\ref{lem:nb-moore} and the preceding degree
bounds give, term by term,
\[
 2(v'+t)\sum_{i=0}^{\ell-1}(v'+2t)^i(v')^{\ell-1-i}
 \le (v')^\ell(v'-1).
\]
Multiplying by \(t\) and using
\[
 2t\sum_{i=0}^{\ell-1}(v'+2t)^i(v')^{\ell-1-i}
 = (v'+2t)^\ell-(v')^\ell
\]
gives
\[
 (v'+t)\bigl((v'+2t)^\ell-(v')^\ell\bigr)
 \le t(v'-1)(v')^\ell,
\]
the reverse of \eqref{eq:ball-condition} for \(v'\), a contradiction.
\end{proof}

\subsection{Spectral cycle exclusion}

The Moore criterion produces a cycle of length \(\le L\) in \(G[S]\).  The
next proposition shows that the spectral certificates forbid such a cycle
whenever the census leaves enough boundary capacity.

\begin{proposition}[Short-cycle exclusion criterion]
\label{lem:cycle-exclusion}
Let \(G\) be a degree-\(3\)-separated obstruction, and let \(L\ge3\) be an
integer.  If \(12L+X\le2n\), then \(G[S]\) contains no cycle of length
\(\le L\).
\end{proposition}

\begin{proof}
Suppose that \(G[S]\) contains a cycle \(C\) of length \(k\le L\), and
put \(A=V(C)\), \(F=N(A)\setminus A\), and
\(B=V\setminus(A\cup F)\).  Write
\(D_C=\sum_{v\in A}d(v)\), \(f=|F|\), and \(b=|B|\).  The cycle uses two
incident edges at each of its \(k\) vertices.  Since every vertex of \(A\)
has degree \(\le4\), and every vertex of \(F\) has a neighbor in \(A\),
\begin{equation}\label{eq:cycle-slice}
  e(A,F)\le D_C-2k\le2k=2|A|,
  \qquad f\le D_C-2k.
\end{equation}

There are \(X+8\) degree-\(3\) vertices, by \eqref{eq:d3-count}, so the
number of degree-\(3\) vertices in \(B\) is \(\le X+8\).  Hence at least
\(b-X-8\) vertices of \(B\) have degree \(\ge4\), and
\(\sum_{v\in B}(d(v)-3)\ge b-X-8\).  Using
\eqref{eq:total-excess} and the exact excess \(D_C-3k\) of \(A\), we obtain
\[
\begin{aligned}
  e(B,F)
  &\le2f+\sum_{v\in F}(d(v)-3)\\
  &\le2f+(n-8)-(D_C-3k)-(b-X-8)\\
  &=4k-D_C+X+3f.
\end{aligned}
\]
Thus \(e(B,F)\le2b\) follows from \(5f+6k+X-D_C\le2n\).  By
\eqref{eq:cycle-slice}, its left-hand side is
\(\le4D_C-4k+X\le12k+X\le12L+X\le2n\).

Since \(k\le L\), the hypothesis also gives \(12k\le12L\le2n-X\le2n\),
so \(n\ge6k\).  As \(f\le2k\), we have
\(b\ge n-3k\ge3k>0\).  Thus \(A\) and \(B\) are nonempty; they are
nonadjacent by the definition of \(F\).  Equation~\eqref{eq:cycle-slice}
and the preceding calculation give their two required boundary bounds.
Proposition~\ref{lem:two-cluster} therefore gives \(\ac(G)\le2\),
contradicting the definition of an obstruction.
\end{proof}

Together with the edge surplus in Lemma~\ref{lem:S-density}, the two cycle
criteria reduce the proof for \(n\ge48\) to finding a length \(L\) for which
\eqref{eq:ball-condition} and \(12L+X\le2n\) hold simultaneously.

\section{Obstruction exclusion for \texorpdfstring{\(n\ge32\)}{n >= 32}}
\label{sec:obstruction-closure}

We now rule out the obstruction class in two overlapping ranges.  For
\(n\ge48\), the exact Moore criterion and the spectral cycle-exclusion
criterion are incompatible.  For \(32\le n\le49\), the incidence census
exceeds the available high-degree capacity.

\subsection{The Moore argument for \texorpdfstring{\(n\ge48\)}{n >= 48}}

We apply the two criteria from Section~\ref{sec:cycles} at the largest cycle
length permitted by the spectral budget.  The census will place the resulting
parameters in the region
\(43\le t\le v\) and \(5v+41\le50\ell+3t\).  The following power estimate
captures the arithmetic needed in that region.

\begin{lemma}[Power growth in the Moore region]\label{lem:moore-power-growth}
Let \(t,v,\ell\) be integers satisfying
\[
 43\le t\le v,
 \qquad 5v+41\le50\ell+3t.
\]
Then the following hold.
\begin{enumerate}[label=\textup{(\roman*)}]
\item If \(4\le\ell\le7\) and \(v\le107\), then
\begin{equation}\label{eq:small-exponent-power}
 (t-9)v^\ell\le(v+2t)^\ell.
\end{equation}
\item If \(\ell\ge8\), then
\begin{equation}\label{eq:large-exponent-power}
 (v+2t)^\ell\ge(t+1)v^\ell.
\end{equation}
\end{enumerate}
\end{lemma}

The proof is given in Appendix~\ref{app:moore-arithmetic}.  Its two parts
reflect different sources of binomial growth.  The case \(\ell=4\) is settled
by a direct endpoint estimate, while \(5\le\ell\le7\) uses the first three
nonconstant terms of the expansion.  Once \(\ell\ge8\), the fourth term alone
proves the stronger estimate in part~\textup{(ii)}.

\begin{proposition}[Uniform Moore arithmetic]\label{prop:moore-arithmetic}
Suppose that \(n,X,h\) are nonnegative integers satisfying
\[
 n\ge48,
 \qquad h\le X,
 \qquad 10X+7h+186\le4n.
\]
Put
\[
 t=n-4-X-3h,\qquad v=n-h,\qquad
 L=\left\lfloor\frac{2n-X}{12}\right\rfloor,\qquad
 \ell=\left\lfloor\frac L2\right\rfloor.
\]
Then \(t\ge43\), \(L\ge8\), and \eqref{eq:ball-condition} holds.
\end{proposition}

\begin{proof}
The nested-floor identity
\(\lfloor\lfloor q/12\rfloor/2\rfloor=\lfloor q/24\rfloor\), applied to
\(q=2n-X\), gives \(\ell=\lfloor(2n-X)/24\rfloor\).  We may therefore write
\begin{equation}\label{eq:moore-remainder}
 2n-X=24\ell+r,\qquad 0\le r\le23.
\end{equation}
The census inequality, together with \(h\le X\), first gives
\[
 4t=4n-16-4X-12h\ge6X-5h+170\ge170,
\]
so \(t\ge43\); also \(v-t=4+X+2h\ge0\).  Moreover,
rearranging the same census inequality gives
\(10(2n-X)\ge16n+7h+186\ge954\).  Hence
\(2n-X\ge96\), \(\ell\ge4\), and \(L\ge8\).

Rewriting the census inequality using \eqref{eq:moore-remainder} gives
\begin{equation}\label{eq:moore-remainder-census}
 8X+7h+186\le48\ell+2r.
\end{equation}
The identity \(50\ell+3t-(5v+41)=26\ell-r-4X-4h-53\) shows that it remains
to prove \(4X+4h+r+53\le26\ell\).  Since \(h\le X\)
and \(r\le23\), \eqref{eq:moore-remainder-census} gives
\(15h\le48\ell-140\le15(4\ell-12)\), where the second inequality uses
\(\ell\ge4\).  Hence \(h\le4\ell-12\).
Using this and \(r\le23\) once more, we obtain
\[
\begin{aligned}
 2(4X+4h+r+53)
 &=8X+8h+2r+106\\
 &\le48\ell+h+4r-80\le52\ell.
\end{aligned}
\]
We have therefore established
\begin{equation}\label{eq:moore-region}
 43\le t\le v,\qquad \ell\ge4,\qquad
 5v+41\le50\ell+3t.
\end{equation}

Suppose first that \(\ell\le7\).  Equation
\eqref{eq:moore-remainder} gives \(2n-X\le191\), or
\(X\ge2n-191\), while the census gives \(10X\le4n-186\).  Therefore
\(20n-1910\le10X\le4n-186\), so \(n\le107\) and \(v\le107\).
Lemma~\ref{lem:moore-power-growth}\textup{(i)} now gives
\((t-9)v^\ell\le(v+2t)^\ell\).  Furthermore,
\[
 (t-10)(v+t)-\bigl(t(v-1)+1\bigr)
 =t^2-9t-10v-1>0,
\]
because its right-hand side is at least
\(43^2-9\cdot43-10\cdot107-1=391\).  It follows that
\[
\begin{aligned}
 t(v-1)v^\ell
 &<\bigl(t(v-1)+1\bigr)v^\ell\\
 &\le(v+t)(t-10)v^\ell\\
 &\le(v+t)\bigl((v+2t)^\ell-v^\ell\bigr),
\end{aligned}
\]
which is \eqref{eq:ball-condition}.

It remains to consider \(\ell\ge8\).  By
Lemma~\ref{lem:moore-power-growth}\textup{(ii)},
\((v+2t)^\ell-v^\ell\ge t v^\ell\), and hence
\[
 t(v-1)v^\ell<(v+t)t v^\ell
 \le(v+t)\bigl((v+2t)^\ell-v^\ell\bigr).
\]
This is again \eqref{eq:ball-condition}.
\end{proof}

\begin{proposition}[No obstruction in the Moore range]
\label{prop:no-obstruction-ge48}
There is no degree-\(3\)-separated obstruction of order \(n\ge48\).
\end{proposition}

\begin{proof}
Suppose that \(G\) is such an obstruction.  Lemma~\ref{lem:moore-census}
and \eqref{eq:h-le-X} verify the hypotheses of
Proposition~\ref{prop:moore-arithmetic}.  Let
\[
  t=n-4-X-3h,\qquad
  v=n-h=|S|,\qquad
  L=\left\lfloor\frac{2n-X}{12}\right\rfloor,\qquad
  \ell=\left\lfloor\frac L2\right\rfloor.
\]
The proposition gives \(t\ge43\), \(L\ge8\), and
\eqref{eq:ball-condition}.  The set \(S\) is nonempty because it contains
the \(8+X\) vertices of degree \(3\), and Lemma~\ref{lem:S-density} gives
\(e(S)\ge|S|+t\).  Proposition~\ref{prop:nb-induced-cycle} therefore
produces a cycle in \(G[S]\) of length \(\le L\).  On the other hand, the
definition of \(L\) gives \(12L+X\le2n\), so
Proposition~\ref{lem:cycle-exclusion} excludes every such cycle.
This contradiction proves the proposition.
\end{proof}

\subsection{The incidence argument for
\texorpdfstring{\(32\le n\le49\)}{32 <= n <= 49}}

To cover the orders below \(48\), we return directly to the incidence census;
the resulting argument is valid throughout \(32\le n\le49\).  Retain
\(X,h,\tau(x),\tau_4\), and \(g\) from Section~\ref{sec:census}.  The vertices
counted by \(g\) are precisely those for which the star certificate does not
impose the bound \(\tau(x)\le d(x)-3\).

\begin{proposition}[Incidence-capacity contradiction]
\label{prop:incidence-capacity-contradiction}
There is no degree-\(3\)-separated obstruction of order
\(32\le n\le49\).
\end{proposition}

\begin{proof}
Suppose that \(G\) is such an obstruction.  We first derive the two capacity
bounds
\begin{align}
  24+2X&\le\tau_4+h+3g,              \label{eq:short-slots}\\
  7\tau_4+3X+32&\le4n.               \label{eq:short-choke}
\end{align}
If \(9d(x)\le n+15\), the star certificate gives
\(\tau(x)\le d(x)-3\); for a vertex counted by \(g\), the trivial bound
costs three additional units.  Thus
\(\sum_{d(x)\ge5}\tau(x)\le X+h+3g\).  Adding the degree-\(4\) contribution
and using \(\sum_{d(x)\ge4}\tau(x)=24+3X\) proves
\eqref{eq:short-slots}.

For \eqref{eq:short-choke}, the inequality \(36\le n+15\) and
Lemma~\ref{lem:star} show that every degree-\(4\) vertex has \(\le1\)
neighbor in \(D_3\).  Hence, for
\(O=\{x:d(x)=4,\ \tau(x)=1\}\), we have \(|O|=\tau_4\).
The set \(O\) is independent.  Indeed, adjacent members with distinct
degree-\(3\) neighbors contradict Lemma~\ref{lem:decorated-edge}(i), while
a shared neighbor creates a \((4,4,3)\)-triangle covered by
Lemma~\ref{lem:cycle-certificate}.

Each member of \(O\) sends its other three edges to vertices of degree
at least \(4\) outside \(O\).  By \eqref{eq:high-degree-sum}, the degree sum
over all vertices of degree at least \(4\) is \(4n-32-3X\); after subtracting
the \(4\tau_4\) contributed by \(O\), we obtain
\(3\tau_4\le4n-32-3X-4\tau_4\), which is \eqref{eq:short-choke}.

Multiply \eqref{eq:short-slots} by \(7\), then use
\eqref{eq:short-choke} and \eqref{eq:h-le-X}.  The result is
\begin{equation}\label{eq:short-aggregate}
  200+10X\le4n+21g.
\end{equation}

If \(39\le n\le49\), then \eqref{eq:exceptional-giants} gives
\(19g\le9X\), hence \(21g\le10X\).  Equation
\eqref{eq:short-aggregate} would force \(200\le4n\), contrary to
\(n\le49\).

Now let \(32\le n\le38\).  Three times
\eqref{eq:exceptional-hoarding}, together with
\eqref{eq:exceptional-giants}, gives
\[
 g(n-29)\le3(n-32)\le2(n-29),
\]
so \(g\le2\).  Equation \eqref{eq:short-aggregate} then gives
\(200\le4n+21g\le194\), again a contradiction.
\end{proof}

The two arguments overlap at orders \(48\) and \(49\) and together exclude
the obstruction class above order \(31\).

\begin{theorem}[Orders at least thirty-two]\label{thm:range-ge32}
If \(n\ge32\) and \(G\) is a graph on \(n\) vertices with
\(2(n-2)\) edges, then \(\ac(G)\le2\).
\end{theorem}

\begin{proof}
Suppose, for a contradiction, that \(\ac(G)>2\).  A vertex of degree
\(\le2\) would have a closed neighborhood of size \(\le3\), whose complement
is nonempty; Lemma~\ref{lem:low-degree} would then give \(\ac(G)\le2\).
Thus \(\delta(G)\ge3\).  Similarly, Lemma~\ref{lem:adjacent-three} shows that
no two degree-\(3\) vertices are adjacent.  Hence \(G\) is a
degree-\(3\)-separated obstruction.  If \(n\le49\), this contradicts
Proposition~\ref{prop:incidence-capacity-contradiction}; if \(n\ge50\), it
    contradicts Proposition~\ref{prop:no-obstruction-ge48}.
\end{proof}

Thus every graph of order \(n\ge32\) satisfies the desired bound.  Below this
threshold, the same degree count still forces a shared degree-\(4\) hub, but
additional local certificates are needed to exclude that configuration.

\section{Orders below thirty-two}
\label{sec:orders-below-thirty-two}

We develop those local certificates and complete the remaining range.
Orders \(4\) through \(9\) follow from low-degree configurations.  For orders
\(10\) through \(31\), a sparse-set argument, supplemented by a short endpoint
census, excludes the shared hub forced by the incidence count.

\subsection{Additional low-order cut certificates}

For a set \(A\subseteq V\), write
\(\partial A=e(A,V\setminus A)\).  We use the following direct
consequences of the variational principle.

\begin{proposition}[Basic cut certificates]\label{lem:cut-certificates}
Let \(G\) have \(n\ge4\) vertices.  Each of the following conditions implies
\(\ac(G)\le2\).
\begin{enumerate}[label=\textup{(\roman*)}]
  \item \(\varnothing\ne A\ne V\) and
  \(n\,\partial A\le2|A|(n-|A|)\).
  \item Disjoint nonempty sets \(A_+,A_-\) have the same cardinality \(k\),
  and, for \(O=V\setminus(A_+\cup A_-)\),
  \(4e(A_+,A_-)+e(A_+,O)+e(A_-,O)\le4k\).
  \item Vertices \(x,y,z\) form a triangle and
  \(n(d(x)+d(y)+d(z)-6)\le6(n-3)\).
  \item Four distinct vertices \(a,b,c,d\) induce exactly the two edges
  \(ab\) and \(cd\), that is, an induced copy of \(2K_2\), and
  \(d(a)+d(b)+d(c)+d(d)\le12\).
\end{enumerate}
\end{proposition}

\begin{proof}
For (i), assign the value \(n-|A|\) on \(A\) and \(-|A|\) on its
complement.  This vector has squared norm \(n|A|(n-|A|)\) and Dirichlet
energy \(n^2\partial A\); hence the hypothesis makes its Rayleigh quotient
\(\le2\), and \eqref{eq:variational} proves (i).  For (ii), use the vector
taking the values \(1,-1,0\) on \(A_+,A_-,O\), respectively.  Its squared
norm is \(2k\), and its Dirichlet energy is the left-hand side of the stated
inequality, so \eqref{eq:variational} likewise proves (ii).  Part (iii)
follows from (i), because the triangle has boundary
\(d(x)+d(y)+d(z)-6\).  For (iv), apply (ii) with \(A_+=\{a,b\}\) and
\(A_-=\{c,d\}\); their total boundary into \(O\) is \(\le12-4=8\).
\end{proof}

The following elementary fact converts the absence of the last two
certificates into a restriction on a low-degree induced subgraph.  It also
follows from the structural results of Chung et
al.~\cite[Corollary to Theorem~1 and Theorem~2]{chung1990}; for completeness,
we give a direct proof for graphs with maximum degree \(\le3\).

\begin{lemma}[Triangle-free induced matching]
\label{lem:triangle-free-induced-matching}
Every triangle-free graph \(J\) on at least eight vertices, with no
isolated vertex and maximum degree \(\le3\), contains an induced
copy of \(2K_2\).
\end{lemma}

\begin{proof}
Suppose otherwise, and choose a vertex \(x\) of maximum degree
\(\Delta\).  Every vertex is at distance \(\le2\) from \(x\).
Indeed, if \(z\) were farther away, choose \(z'\in N(z)\).  The set
\(N(x)\) cannot be contained in \(N(z')\), since
\(z\in N(z')\setminus N(x)\) would give \(d(z')>\Delta\).  An edge
from \(x\) to a vertex in \(N(x)\setminus N(z')\), together with
\(zz'\), would then induce \(2K_2\).

Let \(B\) be the vertices at distance \(2\) from \(x\), and put
\(C_b=N(b)\cap N(x)\) for \(b\in B\).  Each \(C_b\) is nonempty.
Triangle-freeness and the absence of an induced \(2K_2\) imply, for
distinct \(b,b'\in B\),
\[
 C_b\cap C_{b'}=\varnothing\quad\Longleftrightarrow\quad bb'\in E(J),
 \qquad
 bb'\in E(J)\quad\Longrightarrow\quad C_b\cup C_{b'}=N(x).
\]
Every member of \(N(x)\) lies in \(\le\Delta-1\) of the sets
\(C_b\), so \(\sum_{b\in B}|C_b|\le\Delta(\Delta-1)\).  Hence
\(|B|\le2\) if \(\Delta\le2\).  If \(\Delta=3\) and \(|B|\ge4\),
at least two of the \(C_b\) are singletons.  They cannot be distinct,
and if two are equal, their common element already lies in
\(\Delta-1=2\) of the sets \(C_b\).  Every remaining \(C_b\) is therefore
disjoint from that singleton and, by the preceding implications, equals its
complement in \(N(x)\).  A vertex indexed by such a complementary
two-set is adjacent to both singleton occurrences and to two vertices in
\(N(x)\), giving it degree at least \(4\), a contradiction.  Thus
\(|B|\le3\), and in every case
\(|V(J)|\le1+\Delta+|B|\le7\).
\end{proof}

\subsection{Small orders \texorpdfstring{\(4\le n\le9\)}{4 <= n <= 9}}

The preceding cut certificates settle the smallest orders directly.

\begin{proposition}\label{prop:orders-4-9}
If \(4\le n\le9\) and \(G\) has \(n\) vertices and \(2(n-2)\) edges, then
\(\ac(G)\le2\).
\end{proposition}

\begin{proof}
For \(n\le7\), the degree sum satisfies \(4n-8<3n\), so
\(\delta(G)\le2\).  The closed neighborhood of a vertex of degree
\(\le2\) has \(\le3\) vertices, and Lemma~\ref{lem:low-degree}
applies.

It remains to consider \(n=8,9\).  We may again assume
\(\delta(G)\ge3\).  Write \(D_3\) for \(D_3(G)\).  By
\eqref{eq:total-excess},
\(|V\setminus D_3|\le n-8\le1\), so \(|D_3|\ge8\).  Every vertex of
\(G[D_3]\) has degree between \(2\) and \(3\) in \(G[D_3]\).  If
\(G[D_3]\) contains a triangle, its three
vertices all have degree \(3\) in \(G\), and
Proposition~\ref{lem:cut-certificates}(iii) applies.  Otherwise
Lemma~\ref{lem:triangle-free-induced-matching} gives an induced
\(2K_2\) in \(G[D_3]\); its four endpoints have total degree \(12\), so
Proposition~\ref{lem:cut-certificates}(iv) applies.
\end{proof}

\subsection{The shared-hub range \texorpdfstring{\(10\le n\le31\)}{10 <= n <= 31}}

We first show that any obstruction in this range must contain a degree-\(4\)
vertex adjacent to two degree-\(3\) vertices.

\begin{lemma}[Shared-hub existence]\label{lem:shared-hub-existence}
Every degree-\(3\)-separated obstruction of order \(10\le n\le31\)
has a degree-\(4\) vertex with at least two neighbors in \(D_3(G)\).
\end{lemma}

\begin{proof}
Suppose every degree-\(4\) vertex has \(\le1\) such neighbor, and let
\(c\) be the number of vertices of degree at least \(5\).  A degree-\(4\)
vertex contributes \(\le d(x)-3\) to the incidence count, while a
vertex of degree at least \(5\) contributes \(\le(d(x)-3)+3\).
Also \(c\le X\).  Using
\eqref{eq:incidence-total} and
\(\sum_{d(x)\ge4}(d(x)-3)=n-8\) gives
\[
 24+3X
 \le\sum_{d(x)\ge4}(d(x)-3)+3c
 \le n-8+3X,
\]
which forces \(n\ge32\), a contradiction.
\end{proof}

It therefore suffices to exclude this shared-hub configuration.  The
general sparse-set argument reduces its only difficult branch to orders
\(10\) through \(15\); the following elementary incidence fact is needed
at that endpoint.

\begin{lemma}[Six-column incidence lemma]\label{lem:six-column}
Let \(A_0,A_1,A_2\) be disjoint two-element sets and let \(B\) have six
elements.  Suppose \(J\) is a bipartite graph with parts
\(A_0\cup A_1\cup A_2\) and \(B\), and every vertex of \(B\) and both
vertices of \(A_0\) have degree \(3\) in \(J\).  Then two vertices of
\(A_0\cup A_1\cup A_2\) have at least two common neighbors in \(B\),
where either they belong to the same \(A_i\), or one belongs to \(A_0\)
and the other to \(A_1\cup A_2\).
\end{lemma}

\begin{proof}
Suppose not.  For \(b\in B\), put
\(\alpha_b=|N_J(b)\cap A_0|\), \(\beta_b=|N_J(b)\cap A_1|\), and
\(\gamma_b=|N_J(b)\cap A_2|\), and regard
\((\alpha_b,\beta_b,\gamma_b)\) as the column indexed by \(b\).  For a
condition \(P\), let \(\mathbf1_P\) denote its indicator.  Then
\(\alpha_b,\beta_b,\gamma_b\le2\), and
\(\alpha_b+\beta_b+\gamma_b=3\).  Moreover,
\[
 \sum_{b\in B}\alpha_b=6,\qquad
 \sum_{b\in B}\!\bigl(\mathbf1_{\alpha_b=2}
   +\mathbf1_{\beta_b=2}+\mathbf1_{\gamma_b=2}\bigr)\le3,
\]
while
\(\sum_b\alpha_b\beta_b\le4\) and
\(\sum_b\alpha_b\gamma_b\le4\).

Call the column indexed by \(b\) \emph{balanced} if
\((\alpha_b,\beta_b,\gamma_b)=(1,1,1)\), and let \(r_{\mathrm{bal}}\) be the
number of balanced columns.  Every other column has an entry equal to \(2\),
so \(r_{\mathrm{bal}}\ge3\).  Pointwise,
\[
 \mathbf1_{\{b\ {\rm balanced}\}}+2\mathbf1_{\{\gamma_b=0\}}
 \le\alpha_b\beta_b.
\]
After summing, \(r_{\mathrm{bal}}+2|\{b:\gamma_b=0\}|\le4\), and hence
\(\gamma_b\ge1\) for every \(b\).  The same argument with
\(\beta\) and \(\gamma\) interchanged gives \(\beta_b\ge1\).
Consequently \(\alpha_b\le1\), with equality only in a balanced
column.  Thus \(6=\sum_b\alpha_b\le r_{\mathrm{bal}}\), so all six columns are
balanced.  But then \(\sum_b\alpha_b\beta_b=6\), contradicting its
upper bound \(4\).
\end{proof}

We now isolate the finite endpoint of the shared-hub argument.  The
statement includes only the hard branch; the two easy branches will be
handled uniformly afterward.

\begin{proposition}[Endpoint shared-hub census]
\label{prop:endpoint-shared-hub}
Let \(10\le n\le15\), let \(G\) have \(n\) vertices and \(2(n-2)\)
edges, and suppose that \(\delta(G)\ge3\) and \(D_3(G)\) is independent.
Let \(z\) have degree \(4\) and distinct neighbors \(u,v\in D_3(G)\).
Put \(A=\{z,u,v\}\), \(F=N(A)\setminus A\), and
\(B=V\setminus(A\cup F)\).
If \(e(B,F)>2|B|\) and \(I(B)\le2\varepsilon(B)\), then \(\ac(G)\le2\).
\end{proposition}

\begin{proof}
Write \(f=|F|\), \(b=|B|\), \(q=e(B,F)\),
\(I_B=I(B)=2e(B)\), \(I_F=I(F)=2e(F)\), and
\(\varepsilon_F=\varepsilon(F)\), \(\varepsilon_B=\varepsilon(B)\).  Set
\(F_0=N(z)\setminus A\), \(F_1=N(u)\setminus A\), and
\(F_2=N(v)\setminus A\).  These sets all have two elements, and
\(F=F_0\cup F_1\cup F_2\); hence
\(f\le6\).  Every vertex of \(K=F_1\cup F_2\) has degree at least
\(4\), because \(D_3(G)\) is independent.  In particular,
\begin{equation}\label{eq:endpoint-k-excess}
  |K|\le \varepsilon_F.
\end{equation}

The excess on \(A\) is \(1\), and exactly six edges join \(A\) to \(F\).
The corresponding degree and incidence identities are
\begin{align}
 \varepsilon_F+\varepsilon_B&=n-9,                 \label{eq:endpoint-excess}\\
 q+I_B&=3b+\varepsilon_B,                           \label{eq:endpoint-bulk}\\
 I_F+q&=3f+\varepsilon_F-6.                         \label{eq:endpoint-moat}
\end{align}
Also \(q\le \varepsilon_F+2f\).  Combining this bound with
\eqref{eq:endpoint-bulk} and \(I_B\le2\varepsilon_B\) gives
\(3b\le n-9+2f\).  Since \(b=n-3-f\),
\begin{equation}\label{eq:endpoint-f}
  2n\le5f.
\end{equation}

For \(n=10\), \eqref{eq:endpoint-excess} gives
\(\varepsilon_F+\varepsilon_B=1\), contrary to \(|K|\ge2\) and
\eqref{eq:endpoint-k-excess}.  For \(n=11\),
\eqref{eq:endpoint-f} gives \(f\ge5\), so
\(|K|\ge f-|F_0|\ge3>\varepsilon_F+\varepsilon_B=2\), again a contradiction.

For the four remaining orders, equations
\eqref{eq:endpoint-excess}--\eqref{eq:endpoint-f}, the evenness of
\(I_B\), and \eqref{eq:endpoint-k-excess} give precisely the following
profiles.  A tuple records
\((f,b;\varepsilon_F,\varepsilon_B;I_B,q,I_F)\):
\[
\begin{array}{c|c}
n&\text{possible profiles}\\ \hline
12&(5,4;3,0;0,12,0)\\
13&(6,4;4,0;0,12,4)\\
14&(6,5;5,0;0,15,2),\
    (6,5;4,1;0,16,0),\
    (6,5;4,1;2,14,2)\\
15&(6,6;6-j,j;2j,18-j,0),\quad j=0,1,2.
\end{array}
\]
These profiles follow directly from the preceding constraints.  At \(n=12\),
\eqref{eq:endpoint-f} gives \(f\ge5\); if \(f=6\), then
\(|K|\ge4>\varepsilon_F+\varepsilon_B=3\).  Thus \(f=5\), after which
\eqref{eq:endpoint-k-excess} forces \(\varepsilon_F=|K|=3\) and
\(\varepsilon_B=0\).
At \(n=13,14,15\), equation \eqref{eq:endpoint-f} forces \(f=6\).
The three pairs \(F_i\) are then disjoint, so \(|K|=4\) and
\(\varepsilon_B\le n-13\).  Substitution in
\eqref{eq:endpoint-bulk} and \eqref{eq:endpoint-moat}, together with
\(0\le I_B\le2\varepsilon_B\), gives the displayed rows.

It remains to eliminate the profiles.

\smallskip
\noindent\emph{Orders \(12\) and \(13\).}
In both rows, every vertex of \(B\) has degree \(3\).  Moreover,
\(\varepsilon_F=|K|\); since every vertex of \(K\) contributes at least one
unit to \(\varepsilon_F\), no excess remains on \(F\setminus K\).  For
\(n=12\), the equalities
\(|F_0|=2\), \(|K|=3\), and \(|F|=5\) also give \(F_0\cap K=\varnothing\);
for \(n=13\), this follows from the pairwise disjointness noted above.
Hence both vertices of \(F_0\) have degree
\(3\).  They cannot be adjacent to \(B\), so, besides their common
neighbor \(z\), each has two neighbors in \(F\).  This is impossible
when \(n=12\), where \(I_F=0\).  When \(n=13\), two distinct vertices
of internal degree \(2\) require at least three edges in \(G[F]\), and
hence \(I_F\ge6\), contrary to \(I_F=4\).

\smallskip
\noindent\emph{Order \(14\).}
In the first profile, at least one vertex of \(F_0\) has degree \(3\);
it has two neighbors in \(F\), so \(I_F\ge4\), contrary to \(I_F=2\).
In the second profile, both vertices of \(F_0\) have degree \(3\) and
\(I_F=0\).  Each would need two distinct neighbors of degree at least
\(4\) in \(B\), although \(\varepsilon_B=1\) implies
\(|\{x\in B:d(x)\ge4\}|\le1\).
In the last profile, both vertices of \(F_0\) again have degree \(3\),
and \(|\{x\in B:d(x)\ge4\}|\le1\).  Each vertex
of \(F_0\) must therefore meet \(F\).  Since \(I_F=2\), the unique edge
of \(G[F]\) joins the two vertices of \(F_0\).  Together with \(z\)
they form a triangle of degree sum \(4+3+3=10\), and
  Proposition~\ref{lem:cut-certificates}(iii) applies.

\smallskip
\noindent\emph{Order \(15\).}
Here \(I_F=0\), so every vertex of \(F\) has one neighbor in \(A\) and
all its remaining neighbors in \(B\).  If \(\varepsilon_B=1\), then at least one
vertex of \(F_0\) has degree \(3\) and needs two distinct high-degree
neighbors in \(B\), whereas \(\varepsilon_B=1\) permits only one.

If \(\varepsilon_B=2\), both vertices of \(F_0\) have degree \(3\).  The set
\(B_+=\{x\in B:d(x)\ge4\}\) consequently consists of two degree-\(4\)
vertices, and both vertices of \(F_0\) are adjacent to both members of
\(B_+\).  For \(Q=\{z\}\cup F_0\cup B_+\), the five vertices, listed with
\(z\) first, have respectively \(\le2,0,0,2,2\) neighbors outside \(Q\).
Thus \(\partial Q\le6\).
Since \(15\cdot6\le2\cdot5\cdot10\),
Proposition~\ref{lem:cut-certificates}(i) applies.

Finally, suppose \(\varepsilon_B=0\).  Every vertex of \(B\) has degree \(3\).
The absence of edges in \(F\), together with independence of \(D_3(G)\),
forces every vertex of \(F\) to have degree \(4\).  Thus the bipartite
graph between \(F\) and \(B\) is \(3\)-regular on both sides.
Apply Lemma~\ref{lem:six-column} to the three pairs
\(F_0,F_1,F_2\), and denote the two common neighbors in \(B\) by
\(b_1,b_2\).  If the pair \(x,y\) supplied by the lemma lies in one \(F_i\), let
\(p\in A\) be the vertex adjacent to both.  In the order
\((p,x,y,b_1,b_2)\), these vertices have respectively
\(\le2,1,1,1,1\) neighbors outside their set, so
\(\partial\{p,x,y,b_1,b_2\}\le6\).  If instead \(x\in F_0\) and
\(y\in F_1\cup F_2\), let \(p,q\in A\) be adjacent to \(x,y\), respectively.
The vertices \(p,q\) are adjacent, and in the order
\((p,q,x,y,b_1,b_2)\), the six vertices have respectively
\(\le2,1,1,1,1,1\) neighbors outside their set, giving
\(\partial\{p,q,x,y,b_1,b_2\}\le7\).  In the two cases,
\(15\cdot6\le2\cdot5\cdot10\) and
\(15\cdot7\le2\cdot6\cdot9\), respectively, so
Proposition~\ref{lem:cut-certificates}(i) completes the proof.
\end{proof}

\begin{lemma}[Shared degree-\(3\) hub]\label{lem:sharp-shared-hub}
Let \(n\ge10\), let \(G\) have \(n\) vertices and \(2(n-2)\) edges,
and suppose that \(\delta(G)\ge3\) and \(D_3(G)\) is independent.
If a degree-\(4\) vertex has two distinct neighbors in \(D_3(G)\), then
\(\ac(G)\le2\).
\end{lemma}

\begin{proof}
Let the hub be \(z\), let \(u,v\) be the two degree-\(3\) neighbors,
and define \(A,F,B,f,b,q,\varepsilon_F,\varepsilon_B\), and \(I_B\) as in
Proposition~\ref{prop:endpoint-shared-hub}.  The set \(A\) is
two-sparse, \(f\le6\), \(B\ne\varnothing\), and there is no edge from
\(A\) to \(B\).  The same identities give
\[
 q\le \varepsilon_F+2f,\qquad q+I_B=3b+\varepsilon_B,\qquad
 \varepsilon_F+\varepsilon_B=n-9.
\]
If \(q\le2b\), then \(e(A,F)=6=2|A|\), and the two-cluster
certificate applies.  If \(I_B>2\varepsilon_B\),
Proposition~\ref{prop:sparse-set-principle}(ii) produces a two-sparse
nonempty subset of \(B\); together with \(A\), part~(i) gives the
contradiction.

We may therefore assume \(q>2b\) and \(I_B\le2\varepsilon_B\).
As in \eqref{eq:endpoint-f}, these inequalities imply
\(2n\le5f\le30\).  Thus \(n\le15\), and
Proposition~\ref{prop:endpoint-shared-hub} completes the proof.
\end{proof}

\begin{theorem}[Orders below thirty-two]\label{thm:range-4-31}
If \(4\le n\le31\) and \(G\) is a graph on \(n\) vertices with
\(2(n-2)\) edges, then \(\ac(G)\le2\).
\end{theorem}

\begin{proof}
Proposition~\ref{prop:orders-4-9} handles \(n\le9\).  Suppose
\(10\le n\le31\).  If \(G\) has a vertex of degree \(\le2\),
Lemma~\ref{lem:low-degree} applies.  We may therefore assume
\(\delta(G)\ge3\).  If two degree-\(3\) vertices are adjacent, use
Lemma~\ref{lem:adjacent-three}.  Otherwise, a counterexample would be a
degree-\(3\)-separated obstruction.  Lemma~\ref{lem:shared-hub-existence}
produces a degree-\(4\) vertex with two degree-\(3\) neighbors, contrary to
Lemma~\ref{lem:sharp-shared-hub}.
\end{proof}

The ranges established in Sections~\ref{sec:obstruction-closure}
and~\ref{sec:orders-below-thirty-two} are disjoint and exhaustive.

\begin{proof}[Proof of Theorem~\ref{thm:main}]
Apply Theorem~\ref{thm:range-4-31} for \(4\le n\le31\) and
Theorem~\ref{thm:range-ge32} for \(n\ge32\).
\end{proof}

\section{Conclusion}\label{sec:conclusion}

We have proved Kolokolnikov's conjecture for every \(n\ge4\).  The proof begins
with a common spectral reduction: explicit Rayleigh test vectors rule out
local configurations until only a rigid degree-\(3\)-separated obstruction
remains.  Degree counting rules out this obstruction directly for
\(32\le n\le49\), where the required incidences exceed what the higher-degree
vertices can accommodate, and indirectly for \(n\ge48\), where an exact Moore
bound forces a short cycle that the spectral certificates forbid.  At the
remaining small orders, shared-hub and sparse-set arguments produce the
required Rayleigh certificates.

Since \(K_{2,n-2}\) has algebraic connectivity \(2\), it is a maximizer
for every \(n\ge4\).  The proof does not classify all equality cases;
determining whether other maximizers occur, and describing them when they
do, remains a separate extremal problem.
\appendix

\section{Arithmetic details for the Moore closure}
\label{app:moore-arithmetic}

This appendix supplies the endpoint calculations used in
Lemma~\ref{lem:moore-power-growth}.  They verify the uniform power estimates;
the graph-theoretic application remains in Section~\ref{sec:obstruction-closure}.

\subsection{Small exponents}

Put \(M=50\ell+3t-41\).  The hypotheses of
Lemma~\ref{lem:moore-power-growth} give \(5v\le M\), and \(t\le v\) gives
\(5t\le M\).

First suppose that \(\ell=4\), so \(M=159+3t\).  The inequality
\(5t\le M\) gives \(t\le79\).  To verify
\((t-9)M^4\le(M+10t)^4\), set \(u=t-43\), so \(0\le u\le36\).  Direct
expansion gives
\begin{align*}
 &(M+10t)^4-(t-9)M^4\\
 &\quad={}
 u^3\bigl(773272-81u^2-5297u\bigr)
 +83801688u^2+2621645152u+31854951952.
\end{align*}
The coefficient in parentheses decreases on \([0,36]\) and has value
\(477604\) at \(u=36\), so the expression is nonnegative.  Since
\(M,v>0\) and \(5v\le M\),
\[
 t-9\le\left(\frac{M+10t}{M}\right)^4
 \le\left(\frac{5v+10t}{5v}\right)^4
 =\left(\frac{v+2t}{v}\right)^4.
\]
This proves Lemma~\ref{lem:moore-power-growth}\textup{(i)} when \(\ell=4\).

Now suppose that \(5\le\ell\le7\).  We first prove
\begin{align}
 3(t-10)v^3\le{}&6\ell t v^2
 +6\ell(\ell-1)t^2v
 +4\ell(\ell-1)(\ell-2)t^3. \label{eq:small-exponent-cubic}
\end{align}
For real \(x\), define
\begin{align*}
 \Phi(x)={}&30\ell t x^2+150\ell(\ell-1)t^2x
       +500\ell(\ell-1)(\ell-2)t^3\\
       &-3(t-10)x^3.
\end{align*}
The hypotheses give \(5t\le5v\le M\), and
\eqref{eq:small-exponent-cubic} is equivalent to \(\Phi(5v)\ge0\).  On
\([5t,M]\),
\[
 \Phi''(x)=60\ell t-18(t-10)x
 \le t(1320-90t)<0,
\]
so \(\Phi\) is concave and its minimum occurs at an endpoint.

Set \(a=\ell-5\) and \(b=t-43\).  Then \(0\le a\le2\), while
\(t\le v\le107\) gives \(0\le b\le64\).  At the left endpoint,
\[
 \frac{\Phi(5t)}{125t^3}
 =4a^3+54a^2+248a+54+3(79-b)\ge0.
\]
For the other endpoint, define
\begin{align*}
 Q_3(b)&=500b^3+72000b^2+3118500b+44471000,\\
 Q_2(b)&=6450b^3+872400b^2+36222750b+504355800,\\
 Q_1(b)&=23770b^3+3057300b^2+121507590b+1658324560.
\end{align*}
Then
\begin{align*}
 \Phi(M)={}&a^3Q_3(b)+a^2Q_2(b)+aQ_1(b)
 +b^3(10299-81b)\\
 &+2032188b^2+82837380b+1174137072.
\end{align*}
Every term is nonnegative because \(10299-81b\ge5115\).  Thus
\(\Phi(5v)\ge0\), proving \eqref{eq:small-exponent-cubic}.

Finally, the first three nonconstant terms of the binomial expansion give
\begin{align*}
3(v+2t)^\ell\ge{}&3v^\ell+6\ell t v^{\ell-1}
 +6\ell(\ell-1)t^2v^{\ell-2}\\
&+4\ell(\ell-1)(\ell-2)t^3v^{\ell-3}.
\end{align*}
After multiplying \eqref{eq:small-exponent-cubic} by
\(v^{\ell-3}\), the last three terms are at least
\(3(t-10)v^\ell\).  Hence
\(3(v+2t)^\ell\ge3(t-9)v^\ell\), proving part~\textup{(i)}.

\subsection{Large exponents}

Suppose that \(\ell\ge8\), and put
\(c_\ell=\ell(\ell-1)(\ell-2)(\ell-3)\).  We claim that
\begin{equation}\label{eq:fourth-term-region}
 3v^4\le2c_\ell t^3.
\end{equation}
If \(t=43\), then the linear hypothesis gives
\(v\le10\ell+17\).  Writing \(u=\ell-8\), direct expansion gives
\begin{align*}
2c_\ell\cdot43^3-3(10\ell+17)^4
={}&129014u^4+2970364u^3+22976314u^2\\
 &+59988164u+1555677>0.
\end{align*}
Together with \(v\le10\ell+17\), this proves
\eqref{eq:fourth-term-region} when \(t=43\).

Now suppose that \(t\ge44\), and set \(M_s=50\ell+3s-41\) for \(s\ge44\).
The hypotheses give \(5v\le M_t\), while \(t\le v\) gives \(M_t\ge5t\).
At \(t=44\), again writing \(u=\ell-8\),
\begin{align*}
1250c_\ell\cdot44^3-3(50\ell+91)^4
={}&87730000u^4+2031980000u^3+15877835000u^2\\
&+42485217400u+4526254317>0.
\end{align*}
For fixed \(\ell\), the ratio \(M_s^4/s^3\) does not increase with
\(s\) while \(M_s\ge5s\).  Indeed, for \(y=1/s\),
\[
 (1+y)^3-\left(1+\frac{3y}{5}\right)^4
 =\frac{y}{625}\bigl(375+525y+85y^2-81y^3\bigr)\ge0,
\]
because \(0<y\le1\).  Since \(M_{s+1}=M_s+3\), this gives
\[
 \left(\frac{M_{s+1}}{M_s}\right)^4
 \le\left(1+\frac{3}{5s}\right)^4
 \le\left(1+\frac1s\right)^3.
\]
The difference \(M_s-5s=50\ell-41-2s\) decreases with \(s\), so
\(M_t\ge5t\) implies \(M_s\ge5s\) for every integer \(44\le s\le t\).
The preceding comparison says precisely that
\(M_{s+1}^4/(s+1)^3\le M_s^4/s^3\).  Thus the inequality at \(s=44\)
propagates to \(s=t\), giving
\(3M_t^4\le1250c_\ell t^3\).  Since \(5v\le M_t\),
\[
 625\cdot3v^4=3(5v)^4\le3M_t^4\le1250c_\ell t^3.
\]
Division by \(625\) proves \eqref{eq:fourth-term-region}.

The fourth nonconstant term of the binomial expansion now gives
\[
\begin{aligned}
3(v+2t)^\ell
&\ge3v^\ell+2c_\ell t^4v^{\ell-4}\\
&\ge3v^\ell+3t v^\ell
 =3(t+1)v^\ell.
\end{aligned}
\]
This proves Lemma~\ref{lem:moore-power-growth}\textup{(ii)}.

\section{Formalization and AI agent design}
\label{sec:formalization}

Conjecture~\ref{conj:kolokolnikov} has been formalized and proved for every
\(n\ge4\) in the Lean~4 proof assistant~\cite{demoura2021} over the
\emph{mathlib} library~\cite{community2019leanmathematicallibrary}, and the formal development was
generated by \emph{MerLean} \cite{li2026merleanproverrecursiveloopingharness,ren2026merleanagenticframeworkautoformalization}, an autonomous Lean-based
theorem-proving AI agent described below.  The canonical machine-checked
statement reads

\begin{center}
\begin{minipage}{0.96\linewidth}
\small
\begin{verbatim}
theorem ACMax.acmax_conjecture (n : ℕ) (hn : 4 ≤ n) [Nonempty (Fin n)] :
    algConn (completeBipartiteGraph (Fin 2) (Fin (n - 2))) = 2 ∧
      ∀ G : SimpleGraph (Fin n), G.edgeFinset.card = 2 * (n - 2) → algConn G ≤ 2
\end{verbatim}
\end{minipage}
\end{center}

\noindent where \texttt{algConn} denotes the algebraic connectivity
\(\ac\), defined as the second-smallest eigenvalue of the graph
Laplacian using mathlib's spectral theory.  The formal development also
proves the Rayleigh-quotient characterization used throughout this paper.
The use of \texttt{Fin n} merely fixes a labeling of the vertices: every
finite simple graph of order \(n\) has an isomorphic representative on this
type.
Lean reports that the theorem depends only on the three standard axioms
\texttt{propext}, \texttt{Classical.choice}, and \texttt{Quot.sound};
\texttt{native\_decide} is never invoked, so every finite computation
is replayed by the proof kernel itself.  An independent audit with the
\texttt{comparator} tool of the Lean
developers\footnote{\url{https://github.com/leanprover/comparator}}
verified the statement at the level of exported terms, enforced the
axiom whitelist, and replayed the exported proof and all of its dependencies
through two independent kernel implementations, both of which accept it.

No part of the formal development was written by hand.
MerLean is constructed from large-language-model agents
(Claude Code with Fable and Opus, Anthropic) operating
over a persistent plan graph.\footnote{\url{https://github.com/MerLeanProver/MerLean}}
Its guiding design principle is a minimal harness: the system contains
little hand-crafted mathematical or Lean-specific guidance, its
capability being carried instead by its memory.  The memory core
builds on the graph-based architecture
Mem0\(^{g}\)~\cite{chhikara2025mem0buildingproductionreadyai}, whose node-and-edge representation
corresponds naturally to the dependency structure of a Lean
development, with semantic retrieval following the embedding strategy
of LeanSearch~\cite{gao2025semanticsearchenginemathlib4}.  Compile-fix subagents prove scaffolded
files concurrently, and a completed node is admitted to the library
only after a fresh kernel-level check of its statement and axiom
footprint; applied to an open question, the system first formalizes the
statement and then pursues several independent proof strategies in
parallel.  Every numerical certificate was cross-checked by exact
big-integer enumeration before formalization; agent output is treated
throughout as untrusted search, the mathematical guarantee resting
entirely on the kernel checks above.  The system has been released as
open source as a lightweight package intended for use in combination
with other agent skills.

The campaign took eighteen days in total and comprised \(694\) lemma
nodes; its raw output of \(671\) Lean files (approximately
\(240{,}000\) lines) was reduced by dependency analysis to the
published import closure of \(63\) Lean files (the root module and its
\(62\) local dependencies), totaling \(29{,}938\) source lines.  Where per-agent
accounting was enabled, \(214\) compile-fix invocations consumed
\(2.5\times10^{7}\) language-model tokens and \(65\) agent-hours; these
figures are lower bounds.

Every component of the mathematical proof has a formal counterpart: the
certificate toolbox of Section~\ref{sec:certificates}, the global degree
constraints of Section~\ref{sec:census}, the exact cycle criteria and Moore
argument of Sections~\ref{sec:cycles} and~\ref{sec:obstruction-closure}, and
the low-order structural arguments of
Section~\ref{sec:orders-below-thirty-two}.  The presentation here reorganizes
those components as one self-contained mathematical proof.
Theorem~\ref{thm:main} is the mathematical form of the machine-checked
theorem displayed above.

\section*{Data and code availability}

The formal source associated with this result is available at
\url{https://github.com/MerLeanProver/ACMaxConjecture}.
No external dataset is used in this paper.

\printbibliography

\end{document}